\documentclass[10pt]{article}
\usepackage[utf8]{inputenc}
\usepackage{amsfonts,amsmath,amssymb,amsthm}
\usepackage{graphicx}
\usepackage{bbm}
\usepackage{xcolor}
\usepackage[hidelinks]{hyperref}

\RequirePackage{geometry}
\newcommand\blfootnote[1]{%
  \begingroup
  \renewcommand\thefootnote{}\footnote{#1}%
  \addtocounter{footnote}{-1}%
  \endgroup
}

\newcommand{\bR}{\mathbb{R}}
\newcommand{\bE}{\mathbb{E}}
\newcommand{\ep}{\epsilon}
\newcommand{\bI}{\mathbbm{1}}
\newcommand{\bP}{\mathbb{P}}

\newcommand{\CL}{\mathcal{L}}
\newcommand{\CH}{\mathcal{H}}

\newcommand{\comment}[1]{} % multiline comments

\newtheorem{proposition}{Proposition}[section]
\newtheorem{theorem}[proposition]{Theorem}

\newtheorem{corollary}[proposition]{Corollary}
\newtheorem{assumptions}[proposition]{Assumptions}

\title{ On convergence of volume of level sets of stationary smooth Gaussian fields}
\author{Dmitry Beliaev \footnote{Mathematical Institute, University of Oxford, UK } \and Akshay Hegde \footnotemark[1] }
\date{}

\begin{document}

\maketitle

\begin{abstract}
 We prove convergence of Hausdorff measure of level sets of smooth Gaussian fields when the levels converge. Given two coupled stationary fields  $f_1, f_2$ , we estimate the difference of Hausdorff measure of level sets in expectation, in terms of $C^2$-fluctuations of the field $F=f_1-f_2$. The main idea in the proof is to represent difference in volume as an integral of mean curvature using the divergence theorem. 
 This approach is different from using Kac-Rice type formula as main tool in the analysis. 
\end{abstract}

% Blank footnote for Maths subject classification and emails

\blfootnote{Emails: \href{mailto:belyaev@maths.ox.ac.uk}{belyaev@maths.ox.ac.uk}, \href{mailto: hegde@maths.ox.ac.uk}{hegde@maths.ox.ac.uk} }

\section{Introduction}

Smooth Gaussian fields appear naturally in some areas of mathematics and physics. Random plane waves are conjectured to describe local behaviour of particles in chaotic domains, to give an example \cite{Berry_1977}. 
It also interacts well with other topics such as percolation theory \cite{Belyaev_survey}. Random fields have found some applications in areas as diverse as oceanography \cite{oceanography}, cosmology \cite{cosmology}, medical imaging \cite{medical-imaging}.

\vspace{0.3cm}

Studying geometrical and topological properties of the field, especially of level/excursion sets of the field is of great interest. 
Particularly, functionals such as volume of level sets, number connected components of level sets are well studied (see \cite{wigman-nodal-survey-2022},\cite{sodin-lecture-notes}).
In problems involving Gaussian fields, sometimes one needs to compare two fields, say by coupling them, when their laws are close.
Comparing geometric observables are of particular interest. 
We show that, with probability close to one, difference in Hausdorff measures of \emph{nodal sets} (i.e. the zero sets) of coupled fields with `close' laws is small. 
The main idea in the proof is to represent difference in volumes of level sets as an integral of mean curvature of the hypersurface using the divergence theorem. 
This representation is classical in Riemannian geometry and has been used extensively in study of minimal surfaces \cite[Chapter 1]{Lawson-minimal-submanifolds}.
The novelty is to get an average estimate of the difference in volumes in the context of Gaussian fields.
Also, we don't rely on Kac-Rice (or any other variation of co-area formula) for the analysis of volume of level sets, which is a standard tool in this topic.
As a by product, we give an explicit formula for the mean curvature of level sets at a given level. 
We believe that proving convergence in distribution of Hausdorff measure of level sets can be done by following the proof idea of Kac-Rice as presented in \cite[Theorem 6.2]{AW09}. 
But it might not be as straight forward as our proof, and proving other convergences might require some new ideas.

\vspace{0.3cm}

\section{Results}

In this article, we consider smooth Gaussian fields $f: \bR^d \to \bR $ with mild non-degeneracy conditions, of fixed dimension $d \geq 2$.
Call a field \emph{stationary} if the covariance kernel $K(x,y)=\bE[f(x)f(y)]$ is translation invariant. Now for stationary fields, the kernel $K$ is a Fourier transform of a finite symmetric Borel measure $\rho$, called spectral measure (normalised to probability measure in this article).

Fix a domain $D= [-R,R]^d \subset \bR^d$. 
Consider two $C^2$-smooth Gaussian fields $f_1,f_2: \bR^d \to \bR$ and a coupling of the fields $f_1, f_2$, by abuse of notation, such that $F=f_1-f_2$ has the $C^2$-fluctuations 

\[\sigma^2_D : =  \sup_{x \in D} \sup_{|\alpha |\leq 2} \text{Var} [ \partial^{\alpha}F(x)].\]

\begin{assumptions}  Assume that the fields $f_1,f_2$ are 
\begin{enumerate} \label{assumptions}
    \item stationary, $C^2$-smooth a.s.
    \item non-degenerate, i.e. $(f_i,\nabla f_i)$ has density in $\bR^{d+1}$ for $i=1,2$.
    \item Morse functions a.s.
\end{enumerate}
\end{assumptions}

\vspace{0.3cm}

Let $\CL^n$ denote $n$-dimensional Lebesgue measure. Let $\CH^n$ denote the $n$- dimensional Hausdorff measure, which is scaled so that $\CH^n([0,1]^n)=\CL^n([0,1]^n)$.
Note that by Bulinskaya lemma (see \cite[section 5.3]{Nazarov-Sodin-components-appendix}), a.s. nodal sets are sub-manifolds of $\bR^d$ of co-dimension one.
So we interchangeably use the terms volume and Hausdorff measure.

\begin{theorem} \label{main-theorem}
Let $\CH^{d-1}(f_i^{-1}(a))$ denote the volume of level sets in the domain $D$. With the setup as above, we have 
\[
\bE|\CH^{d-1}(f_1^{-1}(0))-\CH^{d-1}(f_2^{-1}(0))| \leq C(f_1,f_2)( \CL^d(D) \sqrt{\log R}) \sigma_D^{1/7}
\]
assuming $\sigma_D$ is small enough (say, $\sigma_D <1$). Here, the constant $C(f_1,f_2)$ depends only on the laws of the fields and not the coupling. 
\end{theorem}

The factor $\sqrt{\log R}$ appearing in the above theorem is from the quantitative version of Kolmogorov's existence theorem for smooth Gaussian fields as stated in \cite[Appendix A]{Nazarov-Sodin-components-appendix}.
Also, the exponent $1/7$ in $\sigma_D^{1/7}$ is not optimal, and can be made close to $1/4$.  We believe optimal exponent of $\sigma_D$ is $1$ due to cancellations in the integral of mean curvature in the bulk. 

\vspace{0.3cm}

We make some comments on the assumptions on the fields. 
We would like to point out that the proof of Theorem \ref{main-theorem} works for non-stationary, non-degenerate fields with positive lower bounds on fluctuations of the field and its derivatives, with suitable modifications. We have not used stationarity in any serious way, but the estimations of the bounds become much easier if we assume stationarity. 
Also, assumption that the fields are a.s. Morse functions is not very restrictive and many interesting non-degenerate fields we know are Morse functions a.s.
It can be shown that stationary fields with spectral measures containing an open sets are Morse a.s. 
If the field is isotropic, then also we can show that the field is Morse a.s.
In particular, random plane wave (RPW) model and Bargmann-Fock field (on $\bR^d$) are Morse a.s. 

\vspace{0.3cm}

One such coupling of fields is available using coupling of white noises (see \cite{Belyaev-Maffucci-coupling}). The coupling as in \cite{Belyaev-Maffucci-coupling} gives the following estimate for the fluctuations of the field $F=f_1-f_2$. We have,

\[\sigma_D^2 \leq C (R^d+1) \inf_{ \rho \in \mathcal{P}(\rho_1, \rho_2)} \int (|s|^2+|t|^2+1)^{2+1} |s-t|^2 d \rho(s,t) \]

where $\mathcal{P}(\rho_1, \rho_2)$ is the space of all symmetric couplings of $\rho_1$ and $\rho_2$ and $C$ is a absolute constant. 

Now by the coupling techniques mentioned above, $\sigma_D$ can be controlled by the transport distances between the measures in the domain (in the general case) or by norm of differences in spectral densities (in special cases).
Let's give an example where this is useful. Recall that the spectral measure of random planar waves is the uniform measure on the unit circle in $\bR^2$. 
We can approximate this measure, in the transport distance mentioned above, by a measure supported on finite points. 
This field corresponds to a finite interference of pure sine waves. So we can quantitative bounds on the difference of lengths.

\vspace{0.3cm}

To prove Theorem \ref{main-theorem}, first we study convergence of volume of level sets using the divergence theorem.
Although expressing change in volume of a hypersurface in normal direction in terms of mean curvature is classical as previously mentioned, we need the version as in Proposition \ref{prop: divergence-thm}.

\vspace{0.3cm}

\begin{proposition} \label{prop: divergence-thm}

Let $f: \bR^d \to \bR$ be a non-degenerate, $C^2$-smooth Gaussian field which is Morse function a.s. Let $\CH^{d-1}(f^{-1}(a))$ denote the volume of level set $f^{-1}(a)$ in $D$. Then, almost surely, we have

\begin{equation} \label{eq: div-thm}
    \CH^{d-1}(f^{-1}(b))-\CH^{d-1}(f^{-1}(a)) = \iint_D \kappa \bI_{f \in [a,b]} d \text{vol} - \oint_{\partial D} \left \langle \dfrac{\nabla f}{|\nabla f|}, \hat{\eta} \right \rangle \bI_{f \in [a,b]} d\text{S}
\end{equation}

where 
\[ \kappa = \text{div} \left (\dfrac{\nabla f}{|\nabla f|} \right )\]
is ($(d-1)$-times) the mean curvature of level set of $f$ at $x$ and $\hat{\eta}$ is the outward unit normal to $\partial D$ . We also have 
\[
\CH^{d-1}(f^{-1}(b)) \to \CH^{d-1}(f^{-1}(a)), \text{ as } b \to a
\]
almost surely and in $L^1$.
\end{proposition}

\vspace{0.3cm}
As a corollary, we get the following formula for the mean curvature of level sets at a given level. Usually, it is hard to get such explicit formula for general fields.

\begin{corollary} \label{corollary} With assumptions as in Theorem \ref{main-theorem} , we have
\[ \bE[\kappa | f=a]= -a \bE[|\nabla f|].  \]
\end{corollary}

\section{Proofs}

\begin{proof}[Proof of Proposition \ref{prop: divergence-thm}]

Note that $f$ has only finitely many critical points in $D$ a.s. 
We prove in subsection \ref{technical-bits} that $\kappa$ as a function on $D$ is integrable almost surely. 
We also can assume that $f$ has no critical points on $\partial D$. This is because of Bulinskaya lemma, since $\partial D $ is $(d-1)$-dimensional and for non-degenerate, smooth Gaussian $f$ the gradient $\nabla f$ has (Gaussian) density on $\bR^d$. 

\vspace{0.3cm}

\textit{Case 1:} $a,b$ are regular values of $f$. 

\vspace{0.3cm}

Let $R'= D \cap f^{-1}[a,b]$ and the unit outward normal $\hat{\eta}= -{\nabla f}/{|\nabla f|}$ on $f^{-1}(a)$, $\hat{\eta}= {\nabla f}/{|\nabla f|}$ on $f^{-1}(b)$ (assuming $a<b$), outward normal on parts of $\partial D \cap f^{-1}(a,b)$. Assume that $R'$ has no critical points of $f$ and we know that $\kappa$ is continuous except at critical points of $f$.
Apply Greens formula for the function ${\nabla f}/{|\nabla f|}$ on $R'$, we get

\begin{equation}
\begin{split} 
    \int_{f^{-1}(b) \cap D} \left \langle  \dfrac{\nabla f}{|\nabla f|}, \hat{\eta}  \right \rangle dS + & \int_{f^{-1}(a) \cap D} \left \langle  \dfrac{\nabla f}{|\nabla f|}, \hat{\eta} \right \rangle dS + \\ & \oint_{\partial D} \left \langle \dfrac{\nabla f}{|\nabla f|}, \hat{\eta} \right \rangle \bI_{f \in [a,b]} dS 
    = \iint_{R'} \text{div} \left ( \dfrac{\nabla f}{|\nabla f|} \right ) d\text{vol}. 
    \end{split}
\end{equation}

But first two terms of LHS of above equation are $\CH^{d-1}(f^{-1}(b)), -\CH^{d-1}(f^{-1}(a))$ respectively. Hence we get equation \eqref{eq: div-thm} in this case.

\vspace{0.3cm}
If $R'$ has critical points of $f$, then the number of critical points has to be finite. Let $\{x_1, x_2, \ldots x_k\}$ be the critical points in $R'$. Now apply the divergence theorem to the field $\nabla f/ |\nabla f|$ on $R' \setminus \cup_j B_{\delta}(x_j)$. 
Letting $\delta \to 0$, and using integrability $\kappa$ on $D$ (see subsection \ref{technical-bits}), we again get equation (\ref{eq: div-thm}). 

\vspace{0.3cm}

\textit{Case 2:} $a$ or $b$ (or both) are critical values of $f$.

\vspace{0.3cm}

First, let us show continuity of volume of level sets at all levels, including at critical values of $f$. Fix a critical value $a$ of $f$.
By Morse lemma, $f$ can be made a quadratic function at a critical point by re-parametrisation. 
Let $p$ be a critical point, then there is a neighborhood $U$ of $p$ and a smooth chart $(y_1,y_2, \ldots, y_d)$ such that $y_i(p)=0$ and 
\[f(y)= f(p) \pm y_1^2 \pm y_2^2 \cdots \pm y_d^2. \]
We know that the volume of level sets of quadratic functions are continuous. So, given a critical point $p$ of $f$ at level $a$, volume of level sets of $f$ in a neighborhood $U$ of $p$ converge when the levels converge to $a$.
When $x_0 \in f^{-1}(a)$ is a regular point, then there exists a neighborhood $U_{x_0}$ such that the volume of level sets are continuous. 
This follows from the implicit function theorem. Now, using compactness of $f^{-1}(a) \cap D$, we get that volume of level sets is continuous at any arbitrary level.

\vspace{0.3cm}

Since the number of critical values of $f$ is finite in $D$, any critical level in $D$ can be approximated by regular levels of $f$ in $D$.
Let $\ep_n$ be a sequence converging to zero such that $(b-\ep_n), (a+\ep_n)$ are sequences of regular values of $f$. 
By continuity of the volume of level sets, we have 
\[\CH^{d-1}(f^{-1}(b))- \CH^{d-1}(f^{-1}(a)) = \lim_{n \to \infty} [L(b- \ep_n)- L(a+ \ep_n)].\]
Using case 1, we have the integral formula for difference of volume of level sets. Note that 

\[\left \langle \dfrac{\nabla f}{|\nabla f|}, \hat{\eta} \right \rangle \bI_{f \in [a + \ep_n,b-\ep_n]} \to \left \langle \dfrac{\nabla f}{|\nabla f|}, \hat{\eta} \right \rangle \bI_{f \in [a,b]}, \]
\[\kappa \bI_{f \in [a+ \ep_n,b-\ep_n]} \to \kappa \bI_{f \in [a,b]} \]
pointwise. Hence by the dominated convergence theorem, we have equation \eqref{eq: div-thm} for case 2 as well. 

\vspace{0.3cm}

We have that $\CH^{d-1}(f^{-1}(b))\to \CH^{d-1}(f^{-1}(a))$ as $b \to a$ a.s. by above discussion of continuity of length w.r.t levels. 
We also have $\bE [\CH^{d-1}(f^{-1}(b))] \to \bE[\CH^{d-1}(f^{-1}(a))]$ when $b \to a$ by Kac-Rice formula. Hence, by Scheffe's lemma, we have $L^1$ convergence. 

\end{proof}

\vspace{0.3cm}

\begin{proof}[Proof of Corollary \ref{corollary}] 
Take expectation to both sides of the equation (\ref{eq: div-thm}). Switching integration and expectation because of Fubini's theorem,  we get 
\[\bE [\CH^{d-1}(f^{-1}(b))]- \bE[\CH^{d-1}(f^{-1}(a))] + \oint_{\partial D} \bE \left [ \left \langle \dfrac{\nabla f}{|\nabla f|}, \hat{\eta} \right  \rangle \bI_{f \in [a,b]} \right ] dS = \iint_{D} \bE[ \kappa \bI_{f \in [a,b]} ] d\text{vol}. \]

Now, dividing above equation by $b-a$ and taking the limit $b \to a$, we get 
\begin{equation} \label{eq:exp-level-volume}
  -a p(a) \CL^d(D) \bE[|\nabla f|]  + \oint_{\partial D} \bE \left [ \left \langle \dfrac{\nabla f}{|\nabla f|}, \hat{\eta} \right \rangle | f=a \right ] p(a) d\text{S} = \iint_{D} \bE[ \kappa| f=a ] p(a) d\text{vol}
  \end{equation}

where $p$ is the pdf of standard Gaussian random variable. 

\vspace{0.3cm}

First, from stationary Kac-Rice formula, we have 
\[ \lim_{b \to a} \dfrac{\bE [\CH^{d-1}(f^{-1}(b))]- \bE[\CH^{d-1}(f^{-1}(a))]}{b-a} = -a p(a) \CL^d(D) \bE[|\nabla f|].\]

Next, from the continuity of the Gaussian regression formula, we get the respective conditional expectations (see \cite[Theorem 3.2]{AW09} for an explanation).
Consider the expression $\bE[ \kappa \bI_{f \in [a,b]} ]$ and write it in the following form,

\[\bE[ \kappa \bI_{f \in [a,b]} ] = \int_a ^b \bE [\kappa | f= u] p(u) du \]

Now note that $\bE [\kappa | f= u]$ is continuous in $u$, hence $ (b-a)^{-1}\bE[ \kappa \bI_{f \in [a,b]} ] \to \bE[\kappa | f=a]$ as $b \to a$. 
By the dominated convergence theorem, we have 

\[ \lim_{b \to a} \dfrac{1}{b-a} \iint_{D} \bE[ \kappa \bI_{f \in [a,b]} ] d\text{vol} =  \iint_{D} \bE[ \kappa| f=a ] p(a) d \text{vol}. \]

A similar argument works for the claim 

\[ \lim_{b \to a} \dfrac{1}{b-a} \oint_{\partial D} \bE \left [ \left \langle \dfrac{\nabla f}{|\nabla f|}, \hat{\eta} \right  \rangle \bI_{f \in [a,b]} \right ] d\text{S} = \oint_{\partial D} \bE \left [ \left \langle \dfrac{\nabla f}{|\nabla f|}, \hat{\eta} \right \rangle | f=a \right ] p(a) d\text{S}. \]

So, we have the equation \eqref{eq:exp-level-volume}. 

\vspace{0.3cm}

Now, we claim that 
\[ \oint_{\partial D} \bE \left [ \left \langle \dfrac{\nabla f}{|\nabla f|}, \hat{\eta} \right \rangle | f=a \right ] d\text{S} = 0.\]

Since $\nabla f$ and $f$ are pointwise independent r.v. and by stationary, integral on a $(d-1)$-dimensional slab in $\partial D$ cancels that from the opposite slab.
Again by stationarity of $\kappa$, the equation (\ref{eq:exp-level-volume}) reduces to $\bE[\kappa | f=a] p(a)= -a \bE[|\nabla f|] p(a) $. Hence we have the Corollary \ref{corollary}.

\end{proof}

\vspace{0.3cm}

\begin{proof}[Proof of Theorem \ref{main-theorem}]
First, observe that $\CH^{d-1}(f^{-1}(a)) \to 0$ almost surely as $a \to \infty$ or as $a \to -\infty$, since probability that $f$ is unbounded on $D$ is zero.
Now, taking difference of equation (\ref{eq: div-thm}) applied to $f_1, f_2$ and taking $b=0$, $a \to -\infty$ we have, 
\begin{equation}  \label{eq: level_diff}
\begin{split}
    \CH^{d-1}(f_1^{-1}(0))-\CH^{d-1}(f_2^{-1}(0))= & \iint_{D} \left [ \kappa_1 \bI _{f_1 \leq 0} -\kappa_2 \bI_{f_2 \leq 0} \right ] d \text{vol}  \\ 
    & - \int_{\partial D} \left [ \left \langle \dfrac{\nabla f_1}{|\nabla f_1|}, \hat{\eta} \right \rangle  \bI_{f_1 \leq 0 } -  \left \langle \dfrac{\nabla f_2}{|\nabla f_2|}, \hat{\eta} \right \rangle  \bI_{f_2 \leq 0 }\right ] dS . 
    \end{split}
\end{equation}

We bound the bulk term and the boundary term of equation (\ref{eq: level_diff}) separately. 

\vspace{0.3cm}

\textbf{Bulk term}: First we have, 

\begin{equation} \label{eq: curv_bound}
\begin{split}
 \left | \int_{D} \left [ \kappa_1 \bI _{f_1 \leq 0} -\kappa_2 \bI_{f_2 \leq 0} \right ] d \text{vol} \right | \leq & \left | \int_{D} (\kappa_1 -\kappa_2) \bI[f_1,f_2 <0] d \text{vol} \right | \\ & +  \left | \int_{D} \kappa_1 \bI [f_1f_2 <0] d \text{vol} \right | + \left | \int_{D} \kappa_2 \bI [f_1f_2 <0] d \text{vol} \right | . 
\end{split}
\end{equation}

For the second term of equation (\ref{eq: curv_bound}) we show that, with probability close to one, $\CL^d(f_1f_2 <0)$ is small and that integral of curvature is bounded with high probability. 

\vspace{0.3cm}

Note that $\bE[|\kappa_1|^{1+\alpha}] < \infty$ for all $ 0< \alpha <1$ (see section \ref{technical-bits}). Take $\alpha= 1/2$ when applying H\"{o}lder inequality in the following computation.  
Given a point $x \in D$, recall that $\kappa_1(x)$ is the curvature of the level set $f^{-1}(c)$, where $x \in f^{-1}(c)$, at $x$. 

\begin{equation} \label{eq: curv-coupling}
\begin{split}
   \left | \bE  \int_D \kappa_1 \bI [f_1f_2<0] d \text{vol}  \right |  & \leq   \bE \left  | \int_D \kappa_1 \bI [f_1f_2<0] d \text{vol} \right | \\
    & \leq  \int_D \bE |  \kappa_1 \bI [f_1f_2<0] | d \text{vol}  \\
    & \leq (\bE|\kappa_1|^{3/2})^{2/3} \int_D \bP[f_1(x)f_2(x)<0]^{1/3} d \text{vol}   \\
    & \leq C_1 \cdot \CL^d(D)\sup _{D} [(\arccos(\rho(x)))^{1/3}]
    \end{split}
\end{equation}

where $\rho(x)$ is the correlation between $f_1(x)$ and $f_2(x)$,and the constant $C_1$ depends only on the law of the fields.
We have used the identity (\ref{eq: Guassian-product}) in the last inequality.
Note that $\arccos(x)= c_1 \sqrt{(1-x)}+ O((1-x)^{3/2})$ near $x=1$, where $c_1$ is a universal constant. 
We have that $ |1-\rho(x)| \leq \sigma^2_D/2 $ for all $x \in D$. Hence we have, 
\begin{equation} \label{eq: correlation-bound}
    \bE \left [ \left  | \int_D \kappa_1 \bI [f_1f_2<0] d \text{vol} \right | \right ] \leq C_2 \CL^d(D) \sigma_D^{1/3} 
\end{equation}
where the constant $C_2$ only depends on the spectral measure. 

\vspace{0.3cm}

Next, we'll bound the term 
\[\bE \left[ \left | \int_{D} (\kappa_1-\kappa_2 )\bI [f_1,f_2 <0] d\text{vol} \right |\right]. \]

We split the computation into two cases: $||\nabla f_i||< \delta $ for one of the $i=1,2$ and $||\nabla f_i||> \delta $ for both $i$'s (for some fixed $\delta>0$).

Now, 
\begin{equation} \label{eq:curv-diff}
    \begin{split}
 \bE \left[ \left | \int_{D} (\kappa_1-\kappa_2 )\bI [||\nabla f_1|| < \delta] d\text{vol} \right |\right] & \leq \int_{D} \bE \left [ |\kappa_1-\kappa_2| \bI [||\nabla f_1|| < \delta ] \right ] d\text{vol} \\
  & \leq (\bE|\kappa_1-\kappa_2|^{4/3})^{3/4} \int_D \bP (||\nabla f_1||^2 < \delta ^2 )^{1/4} d\text{vol} \\
  & \leq C_3 \CL^d(D) \sqrt{\delta}.
    \end{split}
\end{equation}

In the second inequality, we used the fact that curvature has $1+\alpha$ moments for $\alpha \in ([0,1))$ and applied H\"{o}lder's inequality. 
 Observe that $||\nabla f_1||^2$ has bounded pdf around zero, so  $\bP (||\nabla f_1||^2 < \delta ^2 ) = O(\delta^2)$.

Define \[\beta := ||f_1-f_2||_{C^2(D)}. \]
We exploit explicit representation of the curvature \eqref{eq: curvature} in terms of derivatives of the field. Given that $||\nabla f_1||, ||\nabla f_2||> \delta$ we have, 
\[ |\kappa_1-\kappa_2| \leq \dfrac{1}{\delta ^3} (\beta p_1+ \beta^2 p_2 + \beta^3 p_3 )\]
where $p_i$'s are polynomials in the first two derivatives of $f_1$ of degree at most 2. Hence, 
\begin{equation}
    \begin{split}
    \bE \left | \int_{D} (\kappa_1- \kappa_2) \bI [||\nabla f_1||, ||\nabla f_2||> \delta] d\text{vol} \right | & \leq \delta^{-3} \int_D \bE[(\beta p_1+ \beta^2 p_2 + \beta^3 p_3 )] d\text{vol}.
    \end{split}
\end{equation}

Using Cauchy-Schwartz inequality and the fact that laws of the polynomials $p_i$s are translation invariant, we have the following estimate,
 \[ \bE \left | \int_{D} (\kappa_1- \kappa_2) \bI [||\nabla f_1||, ||\nabla f_2||> \delta] d\text{vol} \right | \leq \dfrac{C_4 \CL^d(D)}{\delta^3} (\sqrt{\bE \beta^2} +\sqrt{\bE \beta^4} + \sqrt{\bE \beta^6}). \]
 But by the moment estimates of $\beta$ given in \cite[A.11.1]{Nazarov-Sodin-components-appendix}, the $\bE \beta^2$ term dominates when the coupling of the fields $f_1,f_2$ close. So we have, 
  \begin{equation} \label{eq:curv-diff-2}
  \bE \left | \int_{D} (\kappa_1- \kappa_2) \bI [||\nabla f_1||, ||\nabla f_2||> \delta] d\text{vol} \right | \leq \dfrac{C_5 \CL^d(D)}{\delta^3} (\sqrt{\bE \beta^2}). 
  \end{equation}

\vspace{0.3cm}

\textbf{Boundary term}: We come to the boundary term of equation \eqref{eq: level_diff}. 

\begin{equation} \label{eq: bdry term}
\begin{split}
    \int_{\partial D} \left [ \left \langle \dfrac{\nabla f_1}{|\nabla f_1|}, \hat{\eta} \right \rangle  \bI_{f_1 \leq 0 } -  \left \langle \dfrac{\nabla f_2}{|\nabla f_2|}, \hat{\eta} \right \rangle  \bI_{f_2 \leq 0 }\right ] dS = 
     \int_{\partial D} \left [ \left \langle \dfrac{\nabla f_1}{|\nabla f_1|} - \dfrac{\nabla f_2}{|\nabla f_2|}, \hat{\eta} \right \rangle \bI[f_1,f_2<0] \right ] dS \\
     + \int_{\partial D} \left [  \left \langle \dfrac{\nabla f_1}{|\nabla f_1|}, \hat{\eta}\right \rangle \bI[f_1<0, f_2>0] \right] dS + \int_{\partial D} \left [  \left \langle \dfrac{\nabla f_2}{|\nabla f_2|}, \hat{\eta}\right \rangle \bI[f_2<0, f_1>0] \right] dS
    \end{split}
\end{equation}

The analysis of bounds of first term of RHS of equation \eqref{eq: bdry term} is similar to that of equation \eqref{eq:curv-diff}.
We get that, 
\begin{equation} \label{eq: bdry-bound}
\left | \bE \int_{\partial D} \left [ \left \langle \dfrac{\nabla f_1}{|\nabla f_1|} - \dfrac{\nabla f_2}{|\nabla f_2|}, \hat{\eta} \right \rangle \bI[f_1,f_2<0] \right ] dS \right | \leq C_6 \CL^{d-1}(\partial D) (\delta_1^2+ \bE \beta/\delta_1)  
\end{equation}
for $\delta_1 >0$. 

Now, second term of RHS is bounded by $C \cdot \CL^{d-1}(\partial D \cap \{f_1f_2 <0 \})$ since ${\nabla f_1}/{|\nabla f_1|}$ is unit vector.
By similar argument which lead to equation \eqref{eq: correlation-bound}, we have 
\begin{equation}
    \bE \CL^{d-1}(\partial D \cap \{f_1f_2 <0 \}) \leq C_8 \CL^{d-1}(\partial D) \sigma_D. 
\end{equation}
This is again dominated by the quantity of RHS of equation \eqref{eq: correlation-bound}. 

\vspace{0.3cm} 

\textbf{Analysis of the final bound}: We combine the bounds from \eqref{eq: correlation-bound}, \eqref{eq:curv-diff},\eqref{eq:curv-diff-2}, and \eqref{eq: bdry-bound}. 
Finally, we get 

\[ \bE|\CH^{d-1}(f_1^{-1}(0))-\CH^{d-1}(f_2^{-1}(0))| \leq C \CL^d(D) \left (\sigma_D^{1/3} + \sqrt{\delta}+ \dfrac{\sqrt{\bE\beta^2}}{\delta^3} + \dfrac{\delta_1^2}{R}+ \dfrac{\bE\beta}{\delta_1 R} \right ).\]

Estimates from \cite[A.9, A.11.1]{Nazarov-Sodin-components-appendix} gives us $\bE \beta \leq C_1(R) \sigma_D $ and $\sqrt{\bE \beta^2} \leq C_2(R) \sigma_D$, where we can show that $C_1(R), C_2(R)$ behave like $\sqrt{\log R}$. 
Choosing $\delta= \sigma_D^{2/7}, \delta_1=\sigma_D^{1/2}$, and assuming $\sigma_D$ is small enough we have, 

\[
\bE|\CH^{d-1}(f_1^{-1}(0))-\CH^{d-1}(f_2^{-1}(0))| \leq C(R)\CL^d(D)\sigma_D^{1/7}.
\]

\end{proof}

\subsection{Technical bits} \label{technical-bits}

\textbf{Moments of curvature r.v.}: We show that the $(1+\alpha)$-moments are finite, where $0\leq \alpha <1 $ for the r.v. $\kappa$ of a $C^2$-smooth, non-degenerate, stationary field $f$. 
Observe that

\begin{equation} \label{eq: curvature}
    \kappa= \dfrac{|\nabla f|^2 \text{Tr}(H(f))- \nabla f H(f) \nabla f ^{\text{T}}}{|\nabla f |^{3}}
\end{equation}

where $H(f)$ is the Hessian of the function $f$, by a simple algebraic computation.

\vspace{0.3cm}

First let us prove that $\bE[|\kappa|^{1+\alpha}] < \infty$ for $d=2$ case. The general case follows from similar computation.
Observe that $\mathbf{x}=(x_1,x_2,x_3,x_4,x_5)=(\partial_x f,\partial_y f, \partial_{xx} f, \partial_{xy} f, \partial_{yy} f)$ is a Gaussian vector
and that $(\partial_x f,\partial_y f) \text{ and } (\partial_{xx} f, \partial_{xy} f, \partial_{yy} f) $ are independent, by stationarity of the field $f$. 
Let $\Sigma$ be the covariance matrix of the Gaussian vector $(\partial_x f,\partial_y f)$ and $\bP_1 $ be the law of $(\partial_{xx} f, \partial_{xy} f, \partial_{yy} f)$. Let $\mathbf{x}=(x_1,x_2)$ and $\mathbf{x}'=(x_3,x_4,x_5)$.

\vspace{0.3cm}

So, 
\begin{equation*}
    \begin{split}
        \bE[|\kappa|^{1+\alpha}] = & \dfrac{1}{\sqrt{\text{det}(2 \pi \Sigma)}} \times \\
       & \int_{\bR^5} \left | \dfrac{x_2^2 x_3 -2x_1 x_2 x_4+ x_1^2x_5}{(x_1^2+x_2^2)^{3/2}} \right |^{1+\alpha} \exp{(-1/2(\mathbf{x}^T \Sigma^{-1} \mathbf{x}))} d\mathbf{x} d \bP_1(\mathbf{x}').
    \end{split}
\end{equation*}

By changing the variables to $x_1= r \cos{\theta}, x_2= r \sin{\theta}$ and keeping other variables same, we get ,

\begin{equation*}
    \begin{split}
        \bE[|\kappa|^{1+\alpha}] = & \dfrac{1}{\sqrt{\text{det}(2 \pi \Sigma)}} \times \\ 
        & \int_I r^{-\alpha} \left |  \sin^2{\theta} x_3 -\sin(2 \theta) x_4+ \cos^2 \theta x_5 \right |^{1+\alpha} \exp (-1/2(\mathbf{\Tilde{x}}^T \Sigma^{-1} \mathbf{\Tilde{x}})) dr d\theta d\bP_1(\mathbf{x}')
    \end{split}
\end{equation*}

where $\mathbf{\Tilde{x}}= (r\cos \theta, r \sin \theta)$ and $I = [0, \infty]\times [0, 2 \pi] \times \bR^3$. 
Now, for $0\leq \alpha <1$ the above integral converges. Near the origin of $I$ convergence is taken care by  $\int_0^1 r^{-\alpha} dr < \infty$  and  away from origin  $\exp (\cdots)$ dominates. 
The result follows from the fact that the vector $(\partial_{xx} f, \partial_{xy} f, \partial_{yy} f)$ has all moments finite.

\vspace{0.3cm}

\textbf{Integrability of curvature function}: Consider a deterministic $C^2$-Morse function $f$ on a compact domain $D \subset \bR^d$. As above, at every $x \in D$ which is a regular point of $f$, define $\kappa$ to be the divergence of unit normal of $f$. 

We prove that 
\[ \int_D |\kappa| d \text{vol} < \infty. \]
 Note that except at critical points of $f$, $\kappa$ is continuous. 
So just need to show that $\int_{B_r(x_0)} |\kappa| d \text{vol} < \infty$ for a critical point $x_0$ of $f$ and a small enough ball $B_r(x_0)$ around $x_0$.

\vspace{0.3cm}

We have $\nabla f(x)= H(f)|_{x_0}(x-x_0) + O (||x-x_0||^2)$, by Taylor's series. Since $f$ is Morse, we can invert $H(f)|_{x_0}$ to have 

\[ ||\nabla f(x)|| \geq C \dfrac{||x-x_0||}{||H(f)_{x_0}^{-1}||}. \]

Since $\partial_{xx} f, \partial_{xy} f, \partial_{yy} f$ are all bounded on $D$ and $|\partial_x f(x)| \leq c_1 ||x-x_0||$, $|\partial_y f(x)| \leq c_2 ||x-x_0||$ near $x_0$ and again, exploiting the equation \eqref{eq: curvature}, we have 

\[\int_{B_r(x_0)} |\kappa| d \text{vol} \leq \Tilde{C} \int_{B_r(x_0)} \dfrac{1}{||x-x_0||} d \text{vol}. \]

But we have 
\[ \int_{B_r(x_0)} \dfrac{1}{||x-x_0||} d \text{vol} < \infty\]

for any $d \geq 2$.
This completes the proof that the curvature function is integrable on $D$.

\vspace{0.3cm}

\textbf{Product of Gaussian random variables}: If $X,Y$ are standard Gaussians of correlation $\rho$, then the density of the product $Z=XY$ is given by
\begin{equation}
\label{psi}
\psi_Z(z)=\frac{1}{\pi\sqrt{1-\rho^2}}\exp\left(\frac{\rho z}{1-\rho^2}\right)K_0\left(\frac{|z|}{1-\rho^2}\right)
\end{equation}
where $K_0$ is the Bessel function of the second kind. Integrating $\psi$ over $\mathbb{R}^+$,
\begin{equation} \label{eq: Guassian-product}
\bP(Z>0)=\frac{\pi-\arccos(\rho)}{\pi}
\end{equation}
i.e. this probability is close to $1$ or $0$ for $\rho$ close to $\pm 1$.

\bibliographystyle{amsplain}
\bibliography{main}

\end{document}